\documentclass[11pt]{article} 

\usepackage[utf8]{inputenc} 

\usepackage{geometry} 
\usepackage{graphicx} 

\usepackage{booktabs} 
\usepackage{array} 
\usepackage{paralist} 
\usepackage{verbatim} 
\usepackage{subfig} 

\usepackage{amsthm}
\usepackage{latexsym}
\usepackage{array}
\usepackage{dcolumn}
\usepackage{amsfonts}
\usepackage{bm}

\theoremstyle{plain}
\newtheorem{theorem}{Theorem}
\newtheorem{proposition}[theorem]{Proposition}
\newtheorem{lemma}[theorem]{Lemma}

\newtheorem{corollary}[theorem]{Corollary}

\usepackage{amssymb,amsmath}
\usepackage{caption}
\usepackage{fancyhdr} 
\usepackage{sectsty}
\allsectionsfont{\sffamily\mdseries\upshape} 

\newcommand{\ttt}{ \theta }
\newcommand{\eee}{ \epsilon }
\newcommand{\rere}{\Re}

\title{Extending Landau's Theorem on Dirichlet Series with Non-Negative Coefficients}
\author{Brian N. Maurizi}
\date{} 

\begin{document}

\maketitle

\begin{abstract}

A classical theorem of Landau states that, if an ordinary Dirichlet series has non-negative coefficients, then it has a singularity on the real line at its abscissae of absolute convergence.  In this article, we relax the condition on the coefficients while still arriving at the same conclusion.  Specifically, we write $a_n$ as $|a_n| e^{i \ttt_n}$ and we consider the sequences $\{ \; |a_n| \; \}$ and $\{ \; \cos{\ttt_n} \; \}$.  Let $M \in \mathbb{N}$ be given.  The condition on $\{ \; |a_n| \; \}$ is that, dividing the sequence sequentially into vectors of length $M$, each vector lies in a certain convex cone $B \subset [0,\infty)^M$.   The condition on $\{ \; \cos{\ttt_n} \; \}$ is (roughly) that, again dividing the sequence sequentially into vectors of length $M$, each vector lies in the negative of the polar cone of $B$.  We attempt to quantify the additional freedom allowed in choosing the $\ttt_n$, compared to Landau's theorem.  We also obtain sharpness results.

\end{abstract}

\section{Introduction}

A (ordinary) Dirichlet series is a function of the following form, with $a_n \in \mathbb{C}$:
\begin{equation}\label{DirSer}
f(s) = \sum_{n=1}^{\infty} a_n n^{-s} \qquad s \in \mathbb{C} 
\end{equation}
For $s = \sigma+it \in \mathbb{C}$, we denote the real part of $s$ by $\rere{s}$.  The standard region on which a Dirichlet series might be expected to converge is a right half plane, we denote these by
$$
\Omega_{\sigma} = \{ s \in \mathbb{C} : \rere{s} > \sigma \}
$$
and its closure will be written $\overline{\Omega}_{\sigma}$.  Unlike a power series, a Dirichlet series can converge in an open region without converging absolutely anywhere in that region, for example.  A Dirichlet series has several different ``regions of convergence'' $\Omega_{\sigma}$, with several different abscissae $\sigma$ accordingly.  The abscissae most often considered are:
\begin{align}
\sigma_a &= \inf \{ \sigma : \sum a_n n^{-s} \text{ converges absolutely for } s \in \Omega_\sigma \}  \nonumber \\
\sigma_u &= \inf \{ \sigma : \sum a_n n^{-s} \text{ converges uniformly on } \Omega_{\sigma}  \}  \nonumber \\
\sigma_b &= \inf \{ \sigma : \sum a_n n^{-s} \text{ converges to a bounded function on } \Omega_\sigma \}  \nonumber \\
\sigma_c &= \inf \{ \sigma : \sum a_n n^{-s} \text{ converges for all } s \in \Omega_\sigma \}  \nonumber
\end{align}
From the definitions, it is evident that $\sigma_c \le \sigma_b \le \sigma_u \le \sigma_a$.  It is also a basic result that the function $f$ defined by (\ref{DirSer}) is holomorphic on the open region $\Omega_{\sigma_c}$.  Further relations among these abscissae, the coefficients $\{a_n\}$, and the function $f$ are of considerable interest.  Some of the standard results are the following:
\begin{itemize}
\item{ $\sigma_a - \sigma_c \le 1$ (a basic result), and this is sharp (ex. the alternating zeta function $\sum (-1)^{n+1} n^{-s}$)}
\item{ $\sigma_u = \sigma_b$ (\cite{Bohr_1913_1}), henceforth we will denote this abscissa by $\sigma_b$ }
\item{ $\sigma_a - \sigma_b \le 1/2$ (\cite{Bohr_1913_2}), and this is sharp (\cite{Hille_Bohnenblust}) }
\end{itemize}
For other standard results in analytic number theory and Dirichlet series, we refer the interested reader to \cite{apostol}.

There has been recent interest in applying tools from modern analysis to Dirichlet series (see the survey of Hedenmalm \cite{hedenmalm_survey}).  A short list (non-exhaustive in both topics and articles within those topics) includes the interpolation problem within Hilbert spaces of Dirichlet series (\cite{olsen_seip_interp}), the multiplier algebras of Hilbert spaces of Dirichlet series (\cite{hls}, \cite{mccarthy_03}), Carleson-type theorems for Dirichlet series (\cite{hedenmalm_saksman}, \cite{bayart_konyagin_queffelec_04}), and composition operators on spaces of Dirichlet series (\cite{bayart_et_al_2008}).

We mention the above results for contrast, because our result will be ``classic'' in both statement and proof, and we will investigate Dirichlet series which (among other things) satisfy
\begin{equation}\label{E:sigma_c_sigma_a}
\sigma_a = \sigma_c
\end{equation}
Specifically, we are interested in extending the following theorem of Landau (we will find it convenient to translate and assume $\sigma_a = 0$ for all functions we consider):

\begin{theorem}[E. Landau \cite{Landau_Handbuch}]\label{thm_Landau}

Suppose that $f(s)= \sum a_n n^{-s}$ has abscissa of absolute convergence equal to $0$.  If $a_n \in \mathbb{R}, a_n \ge 0$ for all $n$ then $f$ does not extend holomorphically to a neighborhood of $s = 0$.

\end{theorem}

Logically, the property that must account for the situation $\sigma_c < \sigma_a$ is cancellation among the coefficients $\{a_n\}$.  Therefore, once we strictly limit cancellation among the $\{a_n\}$, (\ref{E:sigma_c_sigma_a}) should follow.  A straightforward way to do this is to require $a_n \ge 0$, and the above theorem confirms this (note that the absence of a holomorphic extension about $s=0$ is stronger than (\ref{E:sigma_c_sigma_a}) ).

It is a natural question to ask whether we could impose less strict conditions on the $\{a_n\}$ and still arrive at the same conclusion.  One would expect that our freedom in choosing the coefficients $\{a_n\}$ will be substantially limited, but can these limitations be quantified in some sense?  Our purpose in this article is to explore these questions.

We wish to mention that there are many interesting conclusions which follow from the assumption ``$a_n \ge 0$,'' the above theorem being but one.  We recall a few of them here.  Define $A_N = \sum_{n=1}^N a_n$.

Suppose $a_n \ge 0$, $\sum a_n = \infty$, and $\sum a_n n^{-s}$ converges for $s \in \Omega_0$ (together, these imply $\sigma_a = 0$).  For arbitrary $t_n \in \mathbb{C}$, consider $\sum a_n t_n n^{-s}$.  To conclude that $\sum a_n t_n n^{-s}$ converges in $\Omega_0$, by a basic result applicable to any Dirichlet series it suffices to assume that $\sum a_n t_n$ converges.  In \cite{Borwein_87}, this is improved in this specific situation; he proves it suffices to show that
$$
(1/A_N) \sum_{n=1}^N a_n t_n
$$
converges as $N \rightarrow \infty$.

Suppose $a_n \ge 0$ and $f(s) = \sum a_n n^{-s}$ has $\sigma_a=0$.  This implies $f$ is log-convex on $(0,\infty)$ (this is due to the log-convexity of each term $n^{-s}$, see the discussion in \cite{Cerone_Dragomir_09} ).  It is also proved in \cite{Kolyada_Leindler_98} that $\| f \|_{L^p(0,\infty)}$ can be estimated above and below by a weighted $l^p$ norm of (modified) dyadic blocks of the $\{a_n\}$, and that $\|f\|_{BMO(0,\infty)}$ can be estimated above and below by another ``dyadic block''-type quantity involving the $\{a_n\}$.

We will obtain an extension of the theorem of Landau, it is an interesting question whether there is perhaps a common thread among more than one of the results mentioned above that would extend the requirement ``$a_n \ge 0$.''

Let us write $a_n = |a_n| e^{i \ttt_n}$.  In section \ref{S:examine_proof} we examine the proof of Landau's theorem, and one notes that the proof can be extended in a straightforward way to obtain 

\begin{theorem}[Landau's Theorem, First Extension]\label{thm_landau_firstextension}

Suppose that $f(s)= \sum a_n n^{-s}$ has abscissa of absolute convergence equal to $0$.  If there exists $\gamma > 0 $ such that $\cos(\ttt_n) \ge \gamma$ for all $n$ then $f$ does not extend holomorphically to a neighborhood of $s =0$.

\end{theorem}

We will develop conditions on the $\{a_n\}$ which are expressed as certain restrictions on the sequence $|a_n|$, and related restrictions on the sequence $\{\cos(\ttt_n)\}$.  We will see that as the restrictions on $|a_n|$ are relaxed, the restrictions on $\cos(\ttt_n)$ become more strict.  The above theorem falls on one end of this spectrum, with no requirements on $|a_n|$ and strict requirements on $\cos(\ttt_n)$.

For $\rho \in (0,\infty)$, let us define $B^{\rho} \subset [0,\infty)^M$ by
\begin{equation}\label{Brho}
B^{\rho} = B^{\rho,M} = \left\{ \beta = (\beta_1 , \ldots,  \beta_M) \in [0,\infty)^M \;  : \beta_1 \le \rho \beta_2 \le \rho^2 \beta_3 \le \cdots \le \rho^{M-1} \beta_M    \right\}
\end{equation}
Note that ``$B^0 = \{ (0,0, \ldots , 0, \mathbb{R}_{\ge0}) \}$'', and ``$B^{\infty} = [0,\infty)^M$''; as $\rho$ proceeds from $0 \rightarrow \infty$, $B^{\rho}$ grows to fill $[0,\infty)^M$.

The standard inner product in Euclidean space will be denoted $x \cdot y$.  We denote the polar cone of a convex cone $C \subset \mathbb{R}^M$ by
$$
C^{\sharp} = \{ x \in \mathbb{R}^M : x \cdot c \le 0 \;\; \forall c \in C\}
$$
We obtain the following result:

\begin{theorem}\label{T}

Suppose that $f(s)= \sum a_n n^{-s}$ has abscissa of absolute convergence equal to $0$.  Write $a_n = |a_n| e^{i \ttt_n}$, and fix $M \in \mathbb{N}$.  Suppose that there exists $ \rho > 0$ and $\gamma >0$ such that, for all $l = 0,1, \ldots$, we have
\begin{align}
( |a_{Ml+1}|,   |a_{Ml+2}|, \ldots,  |a_{Ml+M}| )  &\in B^{\rho}  \label{E:eta_in_Brho} \\
( \cos(\ttt_{Ml+1}), \cos(\ttt_{Ml+2}), \ldots, \cos(\ttt_{Ml+M}) \; )  &\in -\left( B^{\rho} \right)^{\sharp} + \gamma (1,1, \ldots , 1) \label{E:psi_in_Brhosharp}
\end{align}
Then $f$ does not have a holomorphic extension to a neighborhood of $s=0$.

\end{theorem}

Some comments on this result.  First, note that condition (\ref{E:eta_in_Brho}) is not a ``global'' growth or decay condition; with $M=2$ it is satisfied by $\rho, 1, \rho , 1 , \ldots$.  Second, noting that
$$
C_1 \subset C_2 \implies C_2^{\sharp} \subset C_1^{\sharp}
$$
we see that if $\rho$ increases then $B^{\rho}$ becomes larger and therefore $\left( B^{\rho} \right)^{\sharp} $ becomes smaller.  In this sense, (\ref{E:eta_in_Brho}) and (\ref{E:psi_in_Brhosharp}) are ``dual'' to one another; the amount of restriction on $|a_n|$ is inversely proportional to the restriction on $\cos(\ttt_n)$.

In theorem \ref{thm_landau_firstextension}, we saw that with no restrictions on the $|a_n|$ we are free to choose $\ttt_n$ with $\cos(\ttt_n) \in [\gamma,1]$; taking a group of $M$ terms, we are free to choose
$$
( \cos(\ttt_{Ml+1}), \cos(\ttt_{Ml+2}), \ldots, \cos(\ttt_{Ml+M}) \; ) \in [\gamma,1]^M
$$
i.e. the $M$-dimensional volume of the set of admissable values of cosines is less than or equal to $1^M=1$.

In theorem \ref{T}, we have placed restrictions on the sequence $|a_n|$.  Therefore, theorem \ref{T} is only interesting if we can considerably increase the freedom in choosing $\ttt_n$, beyond the amount in theorem \ref{thm_landau_firstextension}.  The following volume estimates demonstrate that this is the case.  We need to consider $\rho \ge 1$ and $\rho <1$ separately, because different constraints will bind in the formation of the set $-\left( B^{\rho} \right)^{\sharp} \cap [-1,1]^M$.

\begin{proposition}\label{P:volume_estimate}
\begin{align}
Vol^{\mathbb{R}^M} \left( -\left( B^{\rho} \right)^{\sharp} \cap [-1,1]^M  \right) &\ge \left[ 1 +  \left( \frac{1 - 1/M}{4 \rho + 1/M} \right) \; \right]^M \qquad &\text{for } \rho \ge 1  \label{E:volume_estimate_ge1} \\
Vol^{\mathbb{R}^M} \left( -\left( B^{\rho} \right)^{\sharp} \cap [-1,1]^M  \right) &\ge  2^{M-1} \left[  1 - \rho / 2 (1-\rho) \; \right]  \qquad  &\text{for } \rho < 1    \label{E:volume_estimate_less1}
\end{align}
\end{proposition}

Note that (\ref{E:volume_estimate_less1}) is only useful if $\rho < 2/3$; it may seem that values of $\rho$ near $1$ have been missed by this proposition.  However, for $\rho < 1$ the set $(B^{\rho})^{\sharp}$, and therefore the volume of the set above, is larger than in the case $\rho = 1$, and so for all $\rho < 1$ we have
$$
Vol^{\mathbb{R}^M} \left( -\left( B^{\rho} \right)^{\sharp} \cap [-1,1]^M  \right)  \ge \left[ 1 +  \left( \frac{1 - 1/M}{4 + 1/M} \right) \; \right]^M
$$
This suffices, because the only points we want to make are the following: 
\begin{itemize}
\item{For any $M \ge 2$ and any $\rho \in (0,\infty)$, the amount of ``freedom'' in choosing $\ttt_n$ is strictly greater than that afforded in theorem \ref{thm_landau_firstextension}.  Indeed, although there is \emph{not} some $a >0$ such that we can freely choose each $\cos(\ttt_n)$ in the interval $(-a,1]$, the amount of freedom we are afforded is equivalent to this.}
\item{As $\rho \rightarrow 0$, the amount of ``freedom'' we are afforded approaches $2^{M-1}$.  This is an ``amount'' of freedom equivalent to the following (although the following is \emph{not} the choice we actually have): Choose a single $\cos(\ttt_{Ml+j}) \in (0,1]$ and then $\ttt_{Ml+j'}$ can be arbitrary for $j' \ne j, \; j' \in \{1, \ldots, M\}$.}
\end{itemize}

We also obtain the following sharpness result:

\begin{proposition}\label{P:sharpness}

(I) ($\gamma>0$ is sharp): For any $M$ and any $\rho \in (0,\infty)$, there exists $\{a_n\}$ such that
\begin{itemize}
\item{$\sum a_n n^{-s}$ has $\sigma_a=0$}
\item{ $( |a_{Ml+1}|,   |a_{Ml+2}|, \ldots,  |a_{Ml+M}| )  \in B^{\rho} $ }
\item{ $( \cos(\ttt_{Ml+1}), \cos(\ttt_{Ml+2}), \ldots, \cos(\ttt_{Ml+M}) \; )  \in -\left( B^{\rho} \right)^{\sharp}$  [This is (\ref{E:psi_in_Brhosharp}) with $\gamma=0$]}
\item{ $\sum a_n n^{-s}$ has a holomorphic extension past $s=0$}
\end{itemize}

(II) ($B^{\rho}, (B^{\rho})^{\sharp}$ is sharp): For any $M$ and any $0 < \rho' < \rho$ there exists $\{a_n\}$ and $\gamma >0$ such that
\begin{itemize}
\item{$\sum a_n n^{-s}$ has $\sigma_a=0$}
\item{ $( |a_{Ml+1}|,   |a_{Ml+2}|, \ldots,  |a_{Ml+M}| )  \in B^{\rho} $ }
\item{ $( \cos(\ttt_{Ml+1}), \cos(\ttt_{Ml+2}), \ldots, \cos(\ttt_{Ml+M}) \; )  \in -\left( B^{\rho'} \right)^{\sharp} + \gamma (1,1, \ldots , 1)$ }
\item{ $\sum a_n n^{-s}$ has a holomorphic extension past $s=0$}
\end{itemize}

\end{proposition}

In section \ref{S:orig_proof}, we review the proof of Landau's theorem.  In section \ref{S:examine_proof}, we examine the proof and broaden the hypotheses.  In section \ref{S:extend_thm}, we obtain conditions on $\{a_n\}$ which imply that these broadened hypotheses are satisfied and thus prove theorem \ref{T}.  In section \ref{S:volume}, we prove the volume estimates in proposition \ref{P:volume_estimate}, and in section \ref{S:Sharpness} we prove the sharpness result in proposition \ref{P:sharpness}.

\section{Proof of Landau's Theorem}\label{S:orig_proof}

Our result will build on a standard proof of Landau's theorem, so we begin by reviewing this proof.

\begin{proof}[Proof of Theorem \ref{thm_Landau}]

We begin by supposing \emph{only} that $f(s) = \sum a_n n^{-s}$ has abscissa of absolute convergence equal to $0$.  The condition $a_n \ge 0$ is not yet assumed; when it is used, we will indicate this explicitly.

For contradiction, we assume that $f$ does extend holomorphically to a neighborhood of $0$; suppose that $f$ is holomorphic on $\mathbb{D}(0, 2\eee)$, $\eee >0$.  We have
\begin{align}
f(s) &= \sum_{n =1}^{\infty} a_n n^{-\eee} n^{-(s-\eee)} \nonumber \\
&= \sum_{n =1}^{\infty} a_n n^{-\eee} \exp( - (s - \eee) \log n) \nonumber \\
&= \sum_{n =1}^{\infty} a_n n^{-\eee} \left\{ \sum_{k = 0}^{\infty} \frac{ (-1)^k (\log n)^k (s - \eee)^k}{k!} \right\} \nonumber \\
&= \sum_{n =1}^{\infty}  \left\{ \sum_{k = 0}^{\infty} a_n n^{-\eee} \frac{ (-1)^k (\log n)^k (s - \eee)^k}{k!} \right\} \nonumber
\end{align}
This double series converges absolutely for $|s - \eee| < \eee$, since the sum of the absolute values can be re-arranged to equal
$$
\sum_{n =1}^{\infty} |a_n | n^{-(\eee - |s - \eee|)} 
$$
which is finite by assumption.  Therefore, we re-arrange the double series to obtain
$$
f(s) = \sum_{k = 0}^{\infty} \left\{ \frac{(-1)^k}{k!} \sum_{n =1}^{\infty}   a_n n^{-\eee} (\log n)^k  \right\} (s - \eee)^k
$$
We see that this is the power series for $f$ about the point $s=\eee$.  We have only asserted the convergence of this power series for $|s-\eee|<\eee$.  However, by the assumption that $f$ is holomorphic on $\mathbb{D}(0,2 \eee)$, it must be the case that this power series in fact converges absolutely for $|s-\eee|<2\eee$ (since $\mathbb{D}(\eee,2\eee) \subset \big( \; \mathbb{D}(0,2\eee) \cup RHP \; \big) $).  Therefore, we have finiteness of the expression
\begin{equation}\label{E:power_series_finite}
\sum_{k = 0}^{\infty}  \left| \frac{(-1)^k}{k!} \sum_{n =1}^{\infty}   a_n n^{-\eee} (\log n)^k  \right| |s - \eee|^k
\end{equation}
for $|s-\eee| < 2\eee$.

We could complete the proof \emph{if} we could obtain finiteness of the expression
\begin{equation}\label{E:double_series_finite}
\sum_{k = 0}^{\infty} \sum_{n =1}^{\infty}   |a_n| n^{-\eee} \frac{ (\log n)^k |s - \eee|^k}{k!} 
\end{equation}
for $|s-\eee| < 2\eee$.  This is because, if (\ref{E:double_series_finite}) were finite, then we could re-arrange (\ref{E:double_series_finite}) to obtain
$$
\sum_{n =1}^{\infty} |a_n | n^{-(\eee - |s - \eee|)} < \infty
$$
for $|s-\eee| < 2\eee$.  This would mean that $\sum a_n n^{-s}$ converges absolutely at $s = - \eee/2$ (for example), a contradiction.

It is here that we use the assumption $a_n \ge 0$.  With this requirement on the $a_n$, we note that
$$
\sum_{k = 0}^{\infty}  \left| \frac{(-1)^k}{k!} \sum_{n =1}^{\infty}   a_n n^{-\eee} (\log n)^k  \right| |s - \eee|^k \; = \; \sum_{k = 0}^{\infty} \sum_{n =1}^{\infty}   |a_n| n^{-\eee} \frac{ (\log n)^k |s - \eee|^k}{k!} 
$$
Therefore, we obtain finiteness of (\ref{E:double_series_finite}) and the proof is complete.

\end{proof}

\section{Examining The Proof}\label{S:examine_proof}

Examining this proof, we see that if we only assume:
\begin{itemize}
\item{ $f$ has abscissa of absolute convergence equal to $0$ }
\item{ $f$ extends holomorphically to $\mathbb{D}(0,2\eee)$ }
\end{itemize}
then (\ref{E:power_series_finite}) is finite for all $s \in \mathbb{D}(\eee, 2 \eee)$.  We will re-write (\ref{E:power_series_finite}) as
\begin{equation}\label{E:power_series_finite2}
\sum_{k = 0}^{\infty} \frac{1}{k!}  \left|  \sum_{n =1}^{\infty}   a_n n^{-\eee} (\log n)^k  \right| \; |s - \eee|^k 
\end{equation}
We obtain a contradiction if we can show that (\ref{E:double_series_finite}) is finite for some $s, \;  |s-\eee| > \eee$.  We will re-write (\ref{E:double_series_finite})  as
\begin{equation}\label{E:double_series_finite2}
\sum_{k = 0}^{\infty} \frac{1}{k!} \left[ \sum_{n =1}^{\infty}   |a_n| n^{-\eee} (\log n)^k \; \right] |s - \eee|^k \; .
\end{equation}
We can prove that $f$ fails to have a holomorphic extension about $s=0$ if, for all sufficiently small $\eee$, the finiteness of (\ref{E:power_series_finite2}) for all $s \in \mathbb{D}(\eee,2 \eee)$ implies the finiteness of (\ref{E:double_series_finite2}) for some $s, |s-\eee| > \eee$.  In other words, we can prove the theorem if the implication (\ref{E:general_suff_condition}) below is true for all sufficiently small $\eee$:
\begin{align}
&\sum_{k = 0}^{\infty} \frac{1}{k!}  \left|  \sum_{n =1}^{\infty}   a_n n^{-\eee} (\log n)^k  \right| \; |s - \eee|^k  < \infty \;\; \text{for all } s \in \mathbb{D}(\eee, 2 \eee) \nonumber \\
& \;\;\;\; \implies \; \sum_{k = 0}^{\infty} \frac{1}{k!} \left[ \sum_{n =1}^{\infty}   |a_n| n^{-\eee} (\log n)^k \; \right] |s - \eee|^k < \infty \;\; \text{for some } s, \;  |s-\eee| > \eee  \label{E:general_suff_condition}
\end{align}
We will investigate a very specific way in which (\ref{E:general_suff_condition}) will be true for all sufficiently small $\eee$.  Specifically, we seek conditions on the $\{a_n\}$ which imply that the ``key'' set of inequalities
\begin{equation}\label{keyineq}
\sum_{n =1}^{\infty}   |a_n| n^{-\eee} (\log n)^k \le C_{\eee} \left|  \sum_{n =1}^{\infty}   a_n n^{-\eee} (\log n)^k  \right| \;\; \forall k \ge 0, \;\;\; C_{\eee} \text{ independant of } k \; .
\end{equation}
holds for all sufficiently small $\eee$.  In principle, one could obtain ``(\ref{E:general_suff_condition}) for all sufficiently small $\eee$'' in other ways, but we will focus on obtaining ``(\ref{keyineq}) for all sufficiently small $\eee$''.

To summarize, we have 

\begin{theorem}[Landau's Theorem, Re-formulated]\label{thm_keyineq}

Suppose that $f(s)= \sum a_n n^{-s}$ has abscissa of absolute convergence equal to $0$.  If 
$$
\sum_{n =1}^{\infty}   |a_n| n^{-\eee} (\log n)^k \le C_{\eee} \left|  \sum_{n =1}^{\infty}   a_n n^{-\eee} (\log n)^k  \right| \;\; \forall k \ge 0, \;\;\; C_{\eee} \text{ independant of } k
$$
holds for all sufficiently small $\eee$, then $f$ does not have a holomorphic extension to a neighborhood of $0$.

\end{theorem}

\section{Extending Landau's Theorem: Conditions on Groups of Terms}\label{S:extend_thm}

In order for (\ref{keyineq}) to hold, it is evident that the arguments of the $a_n$ must be ``aligned'' to some degree.  Our main tool for detecting this alignment will be to examine the real part of $\sum_{n =1}^{\infty}   a_n n^{-\eee} (\log n)^k$.  This will detect alignment that is ``oriented towards the positive real axis'' (by rotation, this is equivalent to alignment that is ``oriented'' in any given direction in the same manner).  One clear extension of Landau's theorem is obtained in this way.

\begin{proof}[Proof of Theorem \ref{thm_landau_firstextension}]

Observe that
\begin{align}
\left| \sum_{n =1}^{\infty}   a_n n^{-\eee} (\log n)^k \right| &\ge \rere \sum_{n =1}^{\infty}   a_n n^{-\eee} (\log n)^k	  \nonumber \\
&= \sum_{n =1}^{\infty}  (\rere \; a_n ) n^{-\eee} (\log n)^k	  \nonumber
\end{align}
We write $a_n = |a_n| e^{i \ttt_n}$.  If we had some $\gamma >0$ such that $\cos(\ttt_n) \ge \gamma \; \forall n$, then we would have $ \rere \; a_n = |a_n| \cos(\ttt_n) \ge \gamma |a_n|$ and therefore
$$
\sum_{n =1}^{\infty}  (\rere \; a_n ) n^{-\eee} (\log n)^k \ge \gamma \sum_{n =1}^{\infty} |a_n| n^{-\eee} (\log n)^k
$$
or
$$
\sum_{n =1}^{\infty} |a_n| n^{-\eee} (\log n)^k \le (1/\gamma) \left| \sum_{n =1}^{\infty}   a_n n^{-\eee} (\log n)^k \right|
$$
We see that $1/\gamma$ is independant of $k$, and therefore we apply Theorem \ref{thm_keyineq} and the proof is complete.
\end{proof}

To obtain Theorem \ref{T}, we employ this method, but apply it to groups of terms instead of single terms.  Fix $M \ge 2$ and write
$$
\sum_{n =1}^{\infty}   a_n n^{-\eee} (\log n)^k = \sum_{l=0}^{\infty} \sum_{j=1}^M a_{Ml+j} (Ml+j)^{-\eee} (\log (Ml+j) \; )^k
$$
which yields
\begin{align}
&\rere{ \sum_{n =1}^{\infty}   a_n n^{-\eee} (\log n)^k }  \nonumber \\
&  = \sum_{l=0}^{\infty} \sum_{j=1}^M \rere{ a_{Ml+j} } (Ml+j)^{-\eee} (\log (Ml+j) \; )^k  \nonumber \\
&  = \sum_{l=0}^{\infty} \sum_{j=1}^M |a_{Ml+j}| \cos(\ttt_{Ml+j}) (Ml+j)^{-\eee} (\log (Ml+j) \; )^k  \nonumber
\end{align}
We develop a condition on the group of coefficients $a_{Ml+1} , \ldots , a_{Ml+M} $ that will imply the existence of some $c_{\eee} > 0$ (independant of $k,l$, and in fact it will be independant of $\eee$) such that
\begin{align}
&\sum_{j=1}^M |a_{Ml+j}| \cos(\ttt_{Ml+j}) (Ml+j)^{-\eee} (\log (Ml+j) \; )^k   \nonumber \\
& \;\; \ge c_{\eee} \Big( \; \sum_{j=1}^M |a_{Ml+j}|  (Ml+j)^{-\eee} (\log (Ml+j) \; )^k  \; \Big) \label{E:M_term_condition}
\end{align}
for all sufficiently small $\epsilon > 0$.  Once (\ref{E:M_term_condition}) holds with $c_{\eee}$ independant of $k,l$, for all sufficiently small $\eee$, we have
$$
\rere{ \sum_{n =1}^{\infty}   a_n n^{-\eee} (\log n)^k } \ge c_{\eee} \sum_{n =1}^{\infty}   |a_n| n^{-\eee} (\log n)^k
$$
and the proof of theorem \ref{T} is complete.

We begin with the RHS of (\ref{E:M_term_condition}).  By Taylor expansion, we write
$$
(Ml+j)^{-\eee} = (Ml)^{-\eee} + A \; , \qquad |A| \le \eee (Ml)^{-\eee} l^{-1}
$$
and therefore we have
\begin{align}
&\sum_{j=1}^M |a_{Ml+j}|  (Ml+j)^{-\eee} (\log (Ml+j) \; )^k  \nonumber \\
& \;\;\;\;\; \le (Ml)^{-\eee} \big( 1 +\eee l^{-1} \big) \sum_{j=1}^M |a_{Ml+j}|  (\log (Ml+j) \; )^k \label{E:RHS_with_epsilon_1st_step}
\end{align}
\emph{Suppose} that the following inequality held for $\gamma$ independant of $l,k$:
\begin{equation}\label{E:ineq_no_epsilons}
\sum_{j=1}^M |a_{Ml+j}|  (\log (Ml+j) \; )^k   \le  \gamma^{-1} \sum_{j=1}^M |a_{Ml+j}| \cos(\ttt_{Ml+j}) (\log (Ml+j) \; )^k 
\end{equation}
Applying the Taylor expansion to the LHS in (\ref{E:M_term_condition}) (estimating $\cos(\ttt) \le 1$), we define
$$
\tilde{A} = \eee (Ml)^{-\eee} l^{-1} \sum_{j=1}^M |a_{Ml+j}| (\log (Ml+j) \; )^k
$$
and we have
\begin{align}
&\sum_{j=1}^M |a_{Ml+j}| \cos(\ttt_{Ml+j}) (Ml+j)^{-\eee} (\log (Ml+j) \; )^k   \nonumber \\
& \;\; \ge (Ml)^{-\eee} \sum_{j=1}^M |a_{Ml+j}| \cos(\ttt_{Ml+j}) (\log (Ml+j) \; )^k - \tilde{A}   \nonumber \\
& \;\; \ge (Ml)^{-\eee} \gamma \sum_{j=1}^M |a_{Ml+j}|  (\log (Ml+j) \; )^k   - \tilde{A}  \qquad \text{ [by (\ref{E:ineq_no_epsilons}) ]} \nonumber \\
& \;\; = (Ml)^{-\eee} \left[ \gamma  - \eee l^{-1} \right]  \sum_{j=1}^M |a_{Ml+j}|  (\log (Ml+j) \; )^k   \nonumber \\
& \;\; \ge \left[ \gamma  - \eee l^{-1} \right] \big( 1 +\eee l^{-1} \big)^{-1} \sum_{j=1}^M |a_{Ml+j}|  (Ml+j)^{-\eee} (\log (Ml+j) \; )^k   \qquad \text{ [by (\ref{E:RHS_with_epsilon_1st_step}) ]}  \nonumber
\end{align}
With $\eee<1$ we have $\left[ \gamma  - \eee l^{-1} \right] \big( 1 +\eee l^{-1} \big)^{-1} \ge \left[ \gamma  - l^{-1} \right] \big( 1 + l^{-1} \big)^{-1}$.  We may assume that $a_n=0$ for all small $n$ (since $\sum_{n=1}^{\infty} a_n n^{-s}$ has a holomorphic extension iff $\sum_{n=N}^{\infty} a_n n^{-s}$ does), and therefore we may assume that we are concerned only with large $l$.  For $l$ large (depending only on $\gamma$) we have $\left[ \gamma  - l^{-1} \right] \big( 1 + l^{-1} \big)^{-1} \ge \gamma/2$, and therefore (\ref{E:M_term_condition}) holds (with $c_{\eee} = \gamma/2$)  for all sufficiently small $\eee$, independant of $k,l$ and we are finished.

Therefore, to prove the theorem, it suffices to show that (\ref{E:ineq_no_epsilons}) holds for some $\gamma>0$, independant of $k,l$.  We focus now on (\ref{E:ineq_no_epsilons}).
Let
\begin{align}
\beta_j = \beta_j^{(k,l)} &= |a_{Ml+j}|  (\log (Ml+j) \; )^k  \label{E:beta_j}
\end{align}
and 
$$
\beta = \beta^{(k,l)} = ( \beta_1^{(k,l)} , \ldots , \beta_M^{(k,l)} )
$$
We abbreviate
\begin{align*}
\psi = \psi^{(l)} &= ( \cos(\ttt_{Ml+1}), \cos(\ttt_{Ml+2}), \ldots, \cos(\ttt_{Ml+M}) \; )  \\
\tilde{\psi} &= \psi - \gamma ( 1,1, \ldots , 1)
\end{align*}
We re-write (\ref{E:ineq_no_epsilons}) as $\beta \cdot \psi \ge \gamma ( \beta \cdot ( 1,1, \ldots , 1) \; )$ or
\begin{equation}\label{E:psi_beta_ge_0}
\beta^{(k,l)} \cdot \tilde{\psi}^{(l)}  \ge 0 \qquad \forall k,l
\end{equation}
Our strategy is as follows: Develop a condition on the $|a_n|$ which implies that $\beta^{(k,l)}$ lies in a particular subset $B$ of $[0,\infty)^M$ for all $k,l$ (i.e. $B$ does not depend on $k,l$).  Then, the condition on the $\ttt_n$ is simply $\tilde{\psi}^{(l)} \in -B^{\sharp}$ and (\ref{E:psi_beta_ge_0}) is satisfied.

We have
$$
\frac{\beta_j^{(k,l)} }{ \beta_{j+1}^{(k,l)} }= \frac{|a_{Ml+j}|}{|a_{Ml+(j+1)}|}  \left( \frac{\log(Ml+j)}{\log(Ml+(j+1) )} \right)^k
$$
Suppose we assume that
$$
\frac{|a_{Ml+j}|}{|a_{Ml+(j+1)}|} \le \rho \qquad \forall l , \; \forall j = 1, \ldots, M-1, \;  \text{for some } \rho \in (0,\infty)
$$
Recalling definition (\ref{Brho}), this can be written  $( |a_{Ml+1}|,   |a_{Ml+2}|, \ldots,  |a_{Ml+M}| )  \in B^{\rho}$.  This implies $\beta_j / \beta_{j+1} \le \rho$ for all $k,l$, or
$$
\beta^{(k,l)} \in B^{\rho} \qquad \forall k,l
$$
The set $B^{\rho}$ meets the requirement of being a proper subset of $[0,\infty)^M$ not depending on $k,l$, therefore this is the condition we seek.  We can now prove theorem \ref{T}.  Suppose $( |a_{Ml+1}|,   |a_{Ml+2}|, \ldots,  |a_{Ml+M}| )  \in B^{\rho}$; this means $\beta^{(k,l)} \in B^{\rho}$.  By definition, the set of $\tilde{\psi}$ which satisfy $\tilde{\psi} \cdot \beta \ge 0$ for all $\beta \in B^{\rho}$ equals $-(B^{\rho})^{\sharp}$.  For $\tilde{\psi}^{(l)} \in -(B^{\rho})^{\sharp}$, we therefore have $\tilde{\psi}^{(l)} \cdot \beta^{(k,l)}  \ge 0 ,\; \forall k,l$.  In other words, (\ref{E:psi_beta_ge_0}) holds, thus (\ref{E:ineq_no_epsilons}) holds, and the proof of Theorem \ref{T} is complete.

\section{Volume Calculation}\label{S:volume}

As we mentioned, theorem \ref{T} is only interesting if the restrictions on $\ttt_n$ are broad enough to be a measurable improvement over the requirement $\cos \ttt_n \ge \gamma$.  We require $\tilde{\psi} \in-(B^{\rho})^{\sharp}$, i.e.
$$
( \cos(\ttt_{Ml+1}) , \ldots , \cos( \ttt_{Ml+M}) \; ) \in -(B^{\rho})^{\sharp} + \gamma \begin{pmatrix}  1 \\ \ldots \\1 \end{pmatrix}
$$ 
We want to answer the question 
$$
\text{``How much freedom do we have in choosing }\cos(\ttt_{Ml+1}) , \ldots , \cos( \ttt_{Ml+M}) \text{ ?''}
$$
One way to answer this is to measure the volume
$$
Vol^{\mathbb{R}^M} \left[ \;  \left( -(B^{\rho})^{\sharp} + \gamma \begin{pmatrix}  1 \\ \ldots \\1 \end{pmatrix} \right) \cap [-1,1]^M  \right]
$$
Since this in continuous in $\gamma$, we will estimate
\begin{equation}\label{E:volume}
Vol^{\mathbb{R}^M} \left[ -(B^{\rho})^{\sharp}  \cap [-1,1]^M  \right]
\end{equation}
First, we obtain a more direct description of $ (B^{\rho})^{\sharp}  $, by writing $B^{\rho}$ as the convex cone generated by a finite point set.

\begin{proposition}

Let $x^{(r)} \in \mathbb{R}^M , r = 1, \ldots , M$ be defined by
$$
x^{(r)} =  ( 0 , \ldots , 0 , \rho^{-r} , \rho^{-(r+1)} , \ldots , \rho^{-M})
$$
Then $B^{\rho}$ equals the positive linear span of the $ \{ x^{(r)} \} $.

\end{proposition}

\begin{corollary}

\begin{equation}\label{E:B_rho_sharp_equations}
(B^{\rho})^{\sharp} = \left\{ y = ( y_1 , \ldots , y_M) : \sum_{j=r}^M \rho^{-j} y_j \le 0  \;\;\; \forall r=1, \ldots, M \right\}
\end{equation}
\end{corollary}

\begin{proof}[Proof of Proposition]

We see that $x^{(r)} \in B^{\rho}$ is clear.  If $\beta \in B^{\rho}$ then
$$
\beta = \rho^1 \beta_1 x^{(1)} + \rho^2 (\beta_2 - \rho^{-1} \beta_1)x^{(2)} + \cdots + \rho^M (\beta_M - \rho^{-1} \beta_{M-1} ) x^{(M)}
$$
Each coefficient is positive, so we have written $B^{\rho}$ as a positive linear combination of the $x^{(r)}$.

\end{proof}

Now, we wish to estimate the expression in (\ref{E:volume}).  The cases $\rho \ge 1, \rho < 1$ are treated separately (since different constraints will bind to form the set $-(B^{\rho})^{\sharp}  \cap [-1,1]^M$ in these two cases).

\subsection*{The case $\rho \ge 1$}

The idea is to exhibit a certain disjoint union of rectangles contained in $-(B^{\rho})^{\sharp} \cap [-1,1]^M$, using the description of $-(B^{\rho})^{\sharp}$ given in (\ref{E:B_rho_sharp_equations}).  This is obtained by bisecting a subinterval of $[-1,1]$ in each coordinate (so we will have $2^M$ rectangles), but the location where the $j$th coordinate is bisected in a particular rectangle will depend on the ``location'' of that rectangle in the coordinates $j' > j$.  The natural order in which to consider the indices will be ``$ M, M-1, \ldots$,'' as we shall see.  An example will clarify this; consider $M=2$.  We have the set
$$
-(B^{\rho})^{\sharp}  = \left\{ (y_1, y_2 ) : \; \rho^{-2} y_2 \ge 0 \; , \;\; \rho^{-1} y_1 + \rho^{-2} y_2 \ge 0 \; \right\}
$$
(In the following discussion we use $2^{-1}$ to denote $1/2$, our aim is to minimize the number of parentheses and improve readability, we apologize for any confusion.)

We divide the $y_2$ coordinate into the ranges $(0,2^{-1}) \; , \; (2^{-1},1)$.  If $y_2 \in (0,2^{-1})$, the ``worst case'' estimate for the range of values of $y_1$ is the trivial one, $y_1 \ge 0$, so we divide the range for $y_1$ into $(0,2^{-1}) \; , \; (2^{-1},1)$ as well.  If $y_2 \in (2^{-1},1)$, we can estimate the range of values of $y_1$ to always contain the interval
$$
(- 2^{-1} \rho^{-1} , 1)
$$
and we evenly divide this interval into two pieces:
\begin{align*}
&  \Big( -2^{-1} \rho^{-1}  \;\;  , \; -2^{-1} \rho^{-1}  + 2^{-1} ( 1 + 2^{-1} \rho^{-1} \; ) \;  \Big)  \\
& \Big( -2^{-1} \rho^{-1}  + 2^{-1} ( 1 + 2^{-1} \rho^{-1} \; )  \;\; , \;  -2^{-1} \rho^{-1}  + 2 \; 2^{-1} ( 1 + 2^{-1} \rho^{-1} \; ) \; \Big)
\end{align*}
To summarize, we obtain
\begin{align*}
-(B^{\rho})^{\sharp} \cap [-1,1]^2 & \supset \\
&(0,2^{-1}) \times (0,2^{-1}) \\
\cup \; &(2^{-1},1) \times (0,2^{-1}) \\
\cup \; &\Big( -2^{-1} \rho^{-1}  \;\;  , \; -2^{-1} \rho^{-1}  + 2^{-1} ( 1 + 2^{-1} \rho^{-1} \; ) \;  \Big)   \times (2^{-1},1) \\
\cup \; &\Big( -2^{-1} \rho^{-1}  + 2^{-1} ( 1 + 2^{-1} \rho^{-1} \; )  \;\; , \;  -2^{-1} \rho^{-1}  + 2 \; 2^{-1} ( 1 + 2^{-1} \rho^{-1} \; ) \; \Big)   \times (2^{-1},1) 
\end{align*}
(a disjoint union of four rectangles).  Using the set-addition notation 
$$
\Bigg( a + (b+c) , a + 2(b+c) \Bigg) = a + (b+c) \Bigg( 1, 2 \Bigg)
$$
(with large delimiters to distinguish the actual interval from parentheses), this can be written
$$
-(B^{\rho})^{\sharp} \cap [-1,1]^2 \supset  \bigcup_{j_1, j_2 = 0}^1  \; - 2^{-1} \rho^{-1} j_2 + 2^{-1} ( 1 + 2^{-1} \rho^{-1} j_2 ) \Bigg( j_1, j_1+1 \Bigg)  \times 2^{-1} \Bigg( j_2 ,  j_2+1 \Bigg) 
$$
The expressions above will soon become cumbersome, so we define the function $P$, for $x_1, \ldots, x_n \in \mathbb{R}$, by
$$
P[ x_1 , \ldots , x_n ] = x_1 ( 1+ x_2 ( 1+ x_3( \ldots +x_{n-1} ( 1 + x_n ) ) \cdots )
$$
(Use of square brackets in the definition of $P$ is again for readability).  We will use the convention that, if $x_1, \ldots , x_n$ is an ``empty list,'' then $P[ x_1 , \ldots , x_n ] =0$.  In addition, we write the set $ (a_1,b_1 ) \times (a_2,b_2) \times \cdots \times (a_n, b_n)$ as
$$
\{ y : y_i \in (a_i, b_i) , \; i = 1, \ldots , n \}
$$
At last, we can write the following for the case $M=2$:
\begin{align*}
&-(B^{\rho})^{\sharp} \cap [-1,1]^2 \supset \\
&\bigcup_{j_1, j_2 = 0}^1  \left\{ y : \; y_k \in \;\;\;\; -P\big[ 2^{-1} \rho^{-1} j_{k+1} \big] + P\big[ 2^{-1}, 2^{-1} \rho^{-1} j_{k+1} \big] \Bigg( j_k, j_k+1 \Bigg) \;\; , \;\;  k = 1, 2 \right\}
\end{align*}
Note that, if $k=2$ then $k+1 = 3$ and ``$2^{-1} \rho^{-1} j_3$'' is an empty list (there is no $j_3$), so 
$$
P\big[ 2^{-1} \rho^{-1} j_{k+1} \big] = 0 \; , \qquad  P\big[ 2^{-1}, 2^{-1} \rho^{-1} j_{k+1} \big] = P\big[ 2^{-1} \big] = 2^{-1}
$$

Applying this idea in dimension $M$, we obtain the following.

\begin{lemma}\label{P:volume_lemma}

For $B^{\rho} = B^{\rho,M}$ ($\rho \ge 1$), we have,
\begin{align*}
&-(B^{\rho})^{\sharp} \cap [-1,1]^M \supset \\
&\bigcup_{j_1, ... , j_M = 0}^1  \Bigg\{ y : \; y_k  \in \;\;\;\; - P\big[ 2^{-1} \rho^{-1} j_{k+1}, 2^{-1} \rho^{-1} j_{k+2},  \ldots , 2^{-1} \rho^{-1} j_M \big]  \\
& \; + P\big[ 2^{-1}, 2^{-1} \rho^{-1} j_{k+1} , 2^{-1} \rho^{-1} j_{k+2},  \ldots , 2^{-1} \rho^{-1} j_M \big] \Bigg( j_k, j_k+1 \Bigg) \;\;\; , \;\;  k = 1, \ldots, M \Bigg\}
\end{align*}
This is a disjoint union.

\end{lemma}

\begin{proof}

We wish to consider just one rectangle from the RHS, so fix $j_1 , \ldots , j_M$.  By picking the left endpoint from $(j_k , j_k+1)$, and noting that
$$
a P\big[x_1, \ldots, x_n\big] = P\big[ a x_1, \ldots , x_n\big]
$$
we have
\begin{equation}\label{E:psi_k}
y_k \ge -P\big[ 2^{-1} \rho^{-1} j_{k+1}, \ldots , 2^{-1} \rho^{-1} j_M \big] +  P\big[ 2^{-1} j_k, 2^{-1} \rho^{-1} j_{k+1} , \ldots , 2^{-1} \rho^{-1} j_M \big]
\end{equation}
Plugging this estimate into the sum $\sum_{j=r}^M \rho^{-j} y_j$, the result telescopes to give just the positive term for $j=r$ and the negative term for $j=M$ (and this $j=M$ term is itself an ``empty list''):
\begin{align*}
\sum_{j=r}^M \rho^{-j} y_j \ge & \rho^{-r} P\big[ 2^{-1} j_r, 2^{-1} \rho^{-1} j_{r+1} , \ldots , 2^{-1} \rho^{-1} j_M \big] - \rho^{-M} P\big[ 2^{-1} \rho^{-1} j_{M+1}, \ldots , 2^{-1} \rho^{-1} j_M \big] \;  \\
= & \rho^{-r} P\big[ 2^{-1} j_r, 2^{-1} \rho^{-1} j_{r+1} , \ldots , 2^{-1} \rho^{-1} j_M \big]
\end{align*}
We see that this is positive, since each expression $2^{-1} \rho^{-1} j_n$ is positive, and therefore by (\ref{E:B_rho_sharp_equations}) the rectangle is contained in $-(B^{\rho})^{\sharp} $.

Next, we prove containment in $[-1,1]^M$.  Note that, because $\rho \ge 1$, we have $0 \le 2^{-1} \rho^{-1} j_n \le 2^{-1}$, and because $P$ is monotone increasing in each coordinate (as long as all coordinates are positive), we have
$$
P\big[ 2^{-1} \rho^{-1} j_{k+1}, \ldots , 2^{-1} \rho^{-1} j_M \big] \le P\big[ 2^{-1}, 2^{-1} , \ldots , 2^{-1} \big] \le 1
$$
and therefore, by (\ref{E:psi_k}), we have $y_k \ge -1$.  

Picking the right endpoint from the interval $(j_k, j_k + 1)$ in lemma \ref{P:volume_lemma}, we have
\begin{align*}
y_k &\le - P\big[ 2^{-1} \rho^{-1} j_{k+1} , \ldots \big] + (j_k+1) P\big[ 2^{-1}, 2^{-1} \rho^{-1} j_{k+1} , \ldots \big] \\
&= - P\big[ 2^{-1} \rho^{-1} j_{k+1} , \ldots \big] + (j_k+1) 2^{-1} \Big( 1+ P\big[ 2^{-1} \rho^{-1} j_{k+1} , \ldots \big] \Big) \\
&\le - P\big[ 2^{-1} \rho^{-1} j_{k+1} , \ldots \big] + \Big( 1+ P\big[ 2^{-1} \rho^{-1} j_{k+1} , \ldots \big] \Big) \\
&= 1
\end{align*}
Lastly, we show the union is disjoint.  Let $(j_1, \ldots , j_M) \ne (j_1', \ldots , j_M')$ and denote the respective rectangles by $R, R'$.  Let $K$ be the largest value $i$ between $1$ and $M$ such that $j_i \ne j_i'$.  WLOG, suppose $j_K =0, j_K' = 1$ (and we have $j_i = j_i'$ for $ i>K$).  Consider the $K$th coordinate.  We see that, by definition of the rectangles in lemma \ref{P:volume_lemma}, the above information on $j_i$ implies that, $\forall y \in  R, \forall y' \in R'$ we have
$$
y_K < y_K'
$$
and this proves disjointness.
\end{proof}

Having lemma \ref{P:volume_lemma}, we can prove the estimate (\ref{E:volume_estimate_ge1}) from proposition \ref{P:volume_estimate}.

\begin{proof}[Proof of  (\ref{E:volume_estimate_ge1})]

To obtain the estimate in  (\ref{E:volume_estimate_ge1}), it remains to sum the volume of the rectangles from lemma \ref{P:volume_lemma}.  Let $R$ be the rectangle corresponding to $(j_1, \ldots , j_M)$, and let $V$ be its volume.  We have
$$
V = \prod_{k=1}^M P\big[ 2^{-1}, 2^{-1} \rho^{-1} j_{k+1} , 2^{-1} \rho^{-1} j_{k+2},  \ldots , 2^{-1} \rho^{-1} j_M \big]
$$
and we have
$$
P\big[ 2^{-1}, 2^{-1} \rho^{-1} j_{k+1} , 2^{-1} \rho^{-1} j_{k+2},  \ldots , 2^{-1} \rho^{-1} j_M \big] \ge 2^{-1} ( 1 + 2^{-1} \rho^{-1} j_{k+1} )
$$
Noting the value of this expression when $k=M$ (namely $2^{-1}$), $V$ is greater than or equal to $2^{-M} \prod_{k=1}^{M-1}  ( 1 + 2^{-1} \rho^{-1} j_{k+1} )$.  By binomial expansion, this is
$$
2^{-M} \sum_{\eee = (\eee_1 , \ldots , \eee_{M-1})} (2 \rho)^{- \sum \eee_i } \;  j_1^{\eee_1} \cdots j_{M-1}^{\eee_{M-1}}
$$
Summing this over the index set, the total volume of all the rectangles is greater than or equal to
\begin{align*}
& \sum_{j_1 , \ldots , j_M = 0}^1  \;\; 2^{-M} \sum_{\eee = (\eee_1 , \ldots , \eee_{M-1})} (2 \rho)^{- \sum \eee_i } \;  j_1^{\eee_1} \cdots j_{M-1}^{\eee_{M-1}}  \\
= \; &2^{-M} \sum_{\eee = (\eee_1 , \ldots , \eee_{M-1})}  (2 \rho)^{- \sum \eee_i } \sum_{j_1 , \ldots , j_M = 0}^1  \;\;  \;  j_1^{\eee_1} \cdots j_{M-1}^{\eee_{M-1}}  \\
\end{align*}
We have
\begin{align*}
\sum_{j_1 , ... , j_M = 0}^1  \;\;  \;  j_1^{\eee_1} \cdots j_{M-1}^{\eee_{M-1}}  &= \left(  \sum_{j_1=0}^1 j_1^{\eee_1} \right)   \left(  \sum_{j_2=0}^1 j_2^{\eee_2} \right)   \cdots \left(  \sum_{j_{M-1}=0}^1  j_{M-1}^{\eee_{M-1}} \right) \left(  \sum_{j_M=0}^1 1 \right) \\
&= 2^{1-\eee_1} 2^{1-\eee_2} \cdots 2^{1-\eee_{M-1}} 2 \\ 
&= 2^M  2^{- \sum \eee_i }  \\ 
\end{align*}
Thus
\begin{align*}
V &\ge 2^{-M}  \sum_{\eee = (\eee_1 , \ldots , \eee_{M-1})}  (2 \rho)^{- \sum \eee_i } \;\; 2^M \;\; 2^{- \sum \eee_i }   \\
&= \sum_{\eee = (\eee_1 , \ldots , \eee_{M-1})}  (4 \rho)^{- \sum \eee_i }  \\
&= (1 + (4 \rho)^{-1} \; )^{M-1}
\end{align*}
It is an elementary exersize to show that
$$
\eee \le \frac{1 - 1/M}{4 \rho + 1/M} \; \implies \; (1 + (4 \rho)^{-1} \; )^{M-1} \ge (1+\eee)^M
$$
and therefore
$$
V \ge \left( 1 + \frac{1 - 1/M}{4 \rho + 1/M} \right)^M
$$
which completes the proof of (\ref{E:volume_estimate_ge1}).

\end{proof}

\subsection*{The case $\rho < 1$}

This proof will be somewhat simpler, we look at the set
$$
[-1,1]^{M-1} \times [0,1]
$$
and prove that $-(B^{\rho})^{\sharp}$ has large intersection with this set.  This is done by taking the first $M-1$ coordinates and splitting each into the cases $[-1,0]$ and $[0,1]$, giving a division of $[-1,1]^{M-1}$ into $2^{M-1}$ cubes of dimension $M-1$ (all side lengths being $1$).  Then, on each cube we find the range of values for the $M$th coordinate which will remain within $-(B^{\rho})^{\sharp}$.

\begin{lemma}
For $\rho < 1$,
\begin{align}
&-(B^{\rho})^{\sharp} \cap [-1,1]^M \supset  \nonumber \\
& \;\;\;\;\;\;\; \bigcup_{j_1, ..., j_{M-1} = 0}^1 \left\{ y : y_{k} \in (-j_k, -j_k+1) \; \text{for } k<M, \;\; y_M \in \left(\sum_{k=1}^{M-1} j_k \rho^{M-k}, 1 \right) \; \right\}  \nonumber
\end{align}

\end{lemma}

We allow for the possibility that the interval $\left(\sum_{k=1}^{M-1} j_k \rho^{M-k}, 1 \right)$ (and thus the corresponding rectangle) is empty.

\begin{proof}[Proof of Lemma]

Suppose $y$ is a member of one of the rectangles on the $RHS$.  (We index by $d$ to avoid $j_j$).  Using (\ref{E:B_rho_sharp_equations}), we calculate: for $r \in \{1, \ldots, M\}$, we have
\begin{align*}
\sum_{d=r}^M \rho^{-d} y_d &= \sum_{d=r}^{M-1} \rho^{-d} y_d  + \rho^{-M} y_M \\
&\ge \sum_{d=r}^{M-1} \rho^{-d} (-j_d) + \rho^{-M} \sum_{d=1}^{M-1} j_d \rho^{M-d}  \\
&\ge \sum_{d=1}^{M-1} \rho^{-d} (-j_d) + \sum_{d=1}^{M-1} j_d \rho^{-d}  \\
&=0
\end{align*}

\end{proof}

Here, disjointness of these rectangles, and containment in $[-1,1]^M$ is clear, so the volume of $-(B^{\rho})^{\sharp} \cap [-1,1]^M$ can be bounded below.  The volume of the rectangle corresponding to $(j_1, \ldots, j_{M-1})$ is greater than or equal to
$$
1 - \sum_{k=1}^{M-1} j_k \rho^{M-k}
$$
(note that this is true when $\sum_{k=1}^{M-1} j_k \rho^{M-k} > 1$ and therefore the rectangle is empty).  So, the volume of all the rectangles together is greater than or equal to
\begin{align*}
\sum_{j_1, ... , j_{M-1} = 0}^1 \left( 1 - \sum_{k=1}^{M-1} j_k \rho^{M-k} \right) &= 2^{M-1} - \sum_{k=1}^{M-1} \rho^{M-k} \sum_{j_1, ... , j_{M-1} = 0}^1 j_k \\
&= 2^{M-1} - \sum_{k=1}^{M-1} \rho^{M-k} ( 2^{M-2} ) \\
&= 2^{M-1} - 2^{M-2} \rho \left( \frac{1-\rho^{M-1}}{1-\rho} \right) \\
&= 2^{M-1} \left[  1 - 2^{-1} \rho \left( \frac{1-\rho^{M-1}}{1-\rho} \right) \; \right] \\
&\ge 2^{M-1} \left[  1 - 2^{-1} \rho / (1-\rho) \; \right] 
\end{align*}
and (\ref{E:volume_estimate_less1}) is proved.

\section{Sharpness}\label{S:Sharpness}

We prove proposition \ref{P:sharpness}, by constructing counterexamples.  Let $\rho \in (0,\infty)$ and $M \ge 2$ be fixed.  All the counterexamples will be of the following form:
\begin{align}
|a_{Ml+j}| &= l^{-1} \rho^{-j}   \nonumber \\
\cos(\ttt_{Ml+j}) = \cos_j &= \lambda \delta_j + \gamma   \nonumber  \\
\sin(\ttt_{Ml+j}) &= (-1)^l \sqrt{ 1 - \cos^2(\ttt_{Ml+j}) }   \label{E:counterexamples}
\end{align}
where $\lambda > 0, \; \delta_j \in [-1,1]$, and $\gamma$ are yet to be determined (subject to the requirement $\lambda \delta_j + \gamma \in [-1,1]$).  We see that our construction already has the following properties:
\begin{itemize}
\item{$\sum a_n n^{-s}$ has $\sigma_a=0$  [This is due to the factor $l^{-1}$]} 
\item{ $( |a_{Ml+1}|,   |a_{Ml+2}|, \ldots,  |a_{Ml+M}| )  \in B^{\rho}$ }
\end{itemize}
We now develop a sufficient condition on $\lambda, \delta_j, \gamma$ under which the sequence of partial sums $\sum_{n=1}^N a_n n^{\eee}$ is a Cauchy sequence for some $\eee>0$, this proves that $\sum a_n n^{-s}$ has a holomorphic extension past $s=0$. 

Consider
$$
\sum_{n=N}^J a_n n^{\eee}
$$
Let $N = M l_0 + j_0, \; J = M l_1 + j_1$ for $l_0, l_1 \ge 0, j_0, j_1 \in \{1, \ldots , M\}$.  This gives
\begin{align}
\sum_{n=N}^J a_n n^{\eee} &= \sum_{j=j_0}^M a_{M l_0 +j} (M l_0 +j)^{\eee} + \sum_{j=1}^{j_1} a_{M l_1 +j} (M l_1 +j)^{\eee}  + \sum_{l=l_0+1}^{l_1-1} \sum_{j=1}^M a_{M l +j} (M l +j)^{\eee} \nonumber \\
&= (I) + (II) + (III)  \label{E:split_sum_anepsilon}
\end{align}
We see that $(I)$ and $(II)$ are bounded in size by a constant times $l_0^{-(1-\eee)}$, which converges to zero as $N \rightarrow \infty$, so we concentrate on $(III)$.  Note that
$$
(Ml+j)^{\eee} = (Ml)^{\eee} + A_{l,j} \; , \qquad |A_{l,j}| \le \eee (Ml)^{\eee} l^{-1}
$$
We have
\begin{align}
\sum_{l=l_0+1}^{l_1-1} \sum_{j=1}^M a_{M l +j} (M l +j)^{\eee} &= \sum_{l=l_0+1}^{l_1-1} l^{-1} \sum_{j=1}^M \rho^{-j} e^{i \ttt_{Ml+j} } ( (Ml)^{\eee} + A_{l,j} )  \nonumber \\
&= \sum_{l=l_0+1}^{l_1-1} l^{-1} \sum_{j=1}^M \rho^{-j} e^{i \ttt_{Ml+j} } A_{l,j}  + \sum_{l=l_0+1}^{l_1-1} l^{-1}  (Ml)^{\eee}  \sum_{j=1}^M \rho^{-j} e^{i \ttt_{Ml+j} }  \nonumber \\
&= (IIIa) + (IIIb)  \label{E:split_subsum_anepsilon}
\end{align}
The first sum, $(IIIa)$, is bounded in size by 
$$
\sum_{l=l_0+1}^{l_1-1} l^{-1} \eee (Ml)^{\eee} l^{-1} \sum_{j=1}^M \rho^{-j}  \le  \eee M^{\eee}  \left( \sum_{j=1}^M \rho^{-j} \right) \sum_{l=l_0+1}^{l_1-1} l^{-(2-\eee)}
$$
and $\sum_{l=l_0+1}^{l_1-1} l^{-(2-\eee)}$ is (part of) the tail of a convergent sum, so it converges to $0$ as $N \rightarrow \infty$.

We have
$$
(IIIb) = \sum_{l=l_0+1}^{l_1-1} l^{-1}  (Ml)^{\eee}  \sum_{j=1}^M \rho^{-j} \left[ \cos(\ttt_{Ml+j})  + i \sin(\ttt_{Ml+j})  \right]   
$$
and since $\cos(\ttt_{Ml+j})$ depends only on $j$, this can be written
\begin{align}
(IIIb) &= M^{\eee} \left( \sum_{j=1}^M \rho^{-j} \cos_j \right)  \sum_{l=l_0+1}^{l_1-1} l^{-(1-\eee)}   \nonumber  \\
& \;\;\; + i M^{\eee} \left( \sum_{j=1}^M \rho^{-j} \sqrt{ 1 - \cos_j^2 }  \right)  \sum_{l=l_0+1}^{l_1-1} (-1)^l l^{-(1-\eee)}  \nonumber \\
& = (IIIb1) + (IIIb2)   \label{E:split_subsubsum_anepsilon}
\end{align}
We see that $ \sum_{l=l_0+1}^{l_1-1} (-1)^{l} l^{-(1-\eee)}$ is (part of) the tail of an alternating series, so $(IIIb2)$ converges to $0$ as $N \rightarrow \infty$.  Therefore, if we have $\sum_{j=1}^M \rho^{-j} \cos_j = 0$, i.e.
\begin{equation}\label{E:cauchy_suff_condition}
\lambda \sum_{j=1}^M \rho^{-j} \delta_j + \gamma \left( \sum_{j=1}^M \rho^{-j}  \right) = 0 
\end{equation}
then we will have $\sum_{j=N}^J a_n n^{\eee} = o(N)$; this is the sufficient condition under which our construction will also satisfy
\begin{itemize}
\item{ $\sum a_n n^{-s}$ has a holomorphic extension past $s=0$}
\end{itemize}

\subsection*{Proposition \ref{P:sharpness} part $(I)$}

Here, we want to find $\{a_n\}$ which satisfy
\begin{itemize}
\item{$\sum a_n n^{-s}$ has $\sigma_a=0$}
\item{ $( |a_{Ml+1}|,   |a_{Ml+2}|, \ldots,  |a_{Ml+M}| )  \in B^{\rho}$ }
\item{ $( \cos(\ttt_{Ml+1}), \cos(\ttt_{Ml+2}), \ldots, \cos(\ttt_{Ml+M}) \; )  \in -\left( B^{\rho} \right)^{\sharp} $  \text{ [this is (\ref{E:psi_in_Brhosharp}) with $\gamma=0$]}   }
\item{ $\sum a_n n^{-s}$ has a holomorphic extension past $s=0$}
\end{itemize}
We choose $\{a_n\}$ as in (\ref{E:counterexamples}), and furthermore we set $\gamma = 0$.  In view of the discussion above, it only remains to prove that we can choose $\lambda, \delta_j \in [-1,1]$ such that the following three properties hold:
\begin{align*}
\lambda \sum_{j=1}^M \rho^{-j} \delta_j &= 0  \qquad \qquad  \text{ [this is (\ref{E:cauchy_suff_condition}) with $\gamma=0$]}  \\
\lambda ( \delta_1 , \ldots , \delta_M ) &\in - (B^{\rho})^{\sharp} \qquad \qquad \text{ [this is (\ref{E:psi_in_Brhosharp}) with $\gamma=0$]}  \\
\lambda \delta_j &\in [-1,1]   \qquad \qquad \text{[this is the requirement $\cos\ttt \in [-1,1]$  ]}
\end{align*}
Evidently, $\lambda$ is irrelevant to the first two properties, so we set it to $1$ (and then the third property is satisfied).  Writing $\delta = (\delta_1 , \ldots , \delta_M)$, and recalling (\ref{E:B_rho_sharp_equations}), the remaining requirements are that there exists $\delta \in [-1,1]^M$ such that
\begin{align*}
\delta \cdot (\rho^{-1} , \ldots , \rho^{-M}) &= 0  \\
\delta \cdot (0 , 0 , \ldots , \rho^{-r} , \ldots , \rho^{-M}) &\ge 0 \;\;\; \forall r = 1 , \ldots , M \\
\end{align*}
This is nothing more than the statement that, in the system of inequalities
$$
x \cdot (0 , 0 , \ldots , \rho^{-r} , \ldots , \rho^{-M}) \ge 0 \;\;\; : \; r = 1 , \ldots , M
$$
the inequality corresponding to $r=1$ does bind at some point.  This is true because the vectors
$$
(0 , 0 , \ldots , \rho^{-r} , \ldots , \rho^{-M}) \ge 0 \;\;\; : \; r = 1 , \ldots , M
$$
are linearly independant.  For a specific example, we could choose 
\begin{align*}
\delta_j &= c (-1)^{M-j} \rho^j  \qquad &\text{[$M$ even]}  \\
\delta_1=0 , \;\;\;\; \delta_j &= c (-1)^{M-j} \rho^j  \;\; j > 1 \qquad &\text{[$M$ odd]}  \\
\end{align*}

\subsection*{Proposition \ref{P:sharpness} part $(II)$}

Fix $M \ge 2$ and $0 < \rho' < \rho$.  Here, we want to find $\{a_n\}$ and $\gamma>0$ which satisfy
\begin{itemize}
\item{$\sum a_n n^{-s}$ has $\sigma_a=0$}
\item{ $( |a_{Ml+1}|,   |a_{Ml+2}|, \ldots,  |a_{Ml+M}| )  \in B^{\rho} $ }
\item{ $( \cos(\ttt_{Ml+1}), \cos(\ttt_{Ml+2}), \ldots, \cos(\ttt_{Ml+M}) \; )  \in -\left( B^{\rho'} \right)^{\sharp} + \gamma (1,1, \ldots , 1)$ }
\item{ $\sum a_n n^{-s}$ has a holomorphic extension past $s=0$}
\end{itemize}
We choose $\{a_n\}$ as in (\ref{E:counterexamples}).  In view of the discussion above, it only remains to prove that we can choose $\lambda, \delta_j \in [-1,1], \gamma>0$ such that the following three properties hold:
\begin{align*}
\lambda \sum_{j=1}^M \rho^{-j} \delta_j + \gamma \left( \sum_{j=1}^M \rho^{-j}  \right) &= 0   \\
\lambda ( \delta_1 , \ldots , \delta_M ) &\in - (B^{\rho'})^{\sharp} \\
\lambda \delta_j + \gamma &\in [-1,1]  
\end{align*}
Evidently, if we find $\delta \in - (B^{\rho'})^{\sharp}$ such that 
$$
\sum_{j=1}^M \rho^{-j} \delta_j < 0
$$
then we can find arbitrarily small values of $\lambda, \gamma$ such that the first property is satisfied, and therefore we can simultaneously satisfy the third property as well.  Therefore, it suffices to find $\delta$ such that
\begin{align}
\delta \in - (B^{\rho'})^{\sharp}   \nonumber \\
\delta \cdot (\rho^{-1} , \ldots , \rho^{-M}) &< 0  \label{E:counterexample_delta}
\end{align}
Intuitively, this is stating a certain ``properness'' of the containment relations
$$
\rho_1 < \rho_2 \;\; \implies \;\;  B^{\rho_1} \subset B^{\rho_2} \;\; \implies \;\; (B^{\rho_1})^{\sharp} \supset (B^{\rho_2})^{\sharp}
$$
The following example suffices: Define 
$$
x = ( - \rho^{-(M-1)} , 0 , 0 , \ldots , 0 , 1)
$$
We have
$$
x \cdot ( \; (\rho')^{-1} , \ldots , (\rho')^{-M}) = (\rho')^{-M} - (\rho')^{-1} \rho^{-(M-1)}
$$
which is greater than $0$, and clearly
$$
x \cdot ( 0 , \ldots , 0 ,\; (\rho')^{-r} , \ldots , (\rho')^{-M}) >0 \qquad \forall r>1
$$
so we have
$$
\sum_{j=r}^M x_j  (\rho')^{-j} > 0  \qquad \forall r \ge 1
$$
Therefore, there exists $\eee$ such that for any $y$, $|y-x|<\eee$, we have
$$
\sum_{j=r}^M y_j  (\rho')^{-j} > 0  \qquad \forall r \ge 1
$$
which implies $y \in -(B^{\rho'})^{\sharp}$.  We selected $x$ to satisfy
$$
x \cdot (\rho^{-1} , \ldots , \rho^{-M}) = 0
$$
This means (since a non-zero linear functional on $\mathbb{R}^M$ is an open mapping) we have $\delta$, $|\delta - x|<\eee$ such that
$$
\delta \cdot (\rho^{-1} , \ldots , \rho^{-M}) < 0
$$
and thus $\delta$ satisfies (\ref{E:counterexample_delta}), concluding the proof.

\section*{Acknowledgements}

The author would like to thank the MIT Open Courseware Project; this article was conceived while perusing the Analytic Number Theory lecture notes of Prof. Kiran Kedlaya.

\end{document}